\documentclass[reqno]{amsart}

\usepackage[a4paper,margin=1.5in]{geometry}

\usepackage{mathrsfs}
\usepackage{amsfonts}
\usepackage{pdfpages}
\usepackage{amsmath}
\usepackage{cases}
\usepackage{amssymb}
\usepackage[foot]{amsaddr}

\usepackage{mathtools}
\usepackage[inline]{enumitem}
\usepackage{hyperref}

\usepackage{tikz}
\usepackage{tikz-cd}
\usetikzlibrary{arrows,automata}

\newcommand{\pres}[3]{\textnormal{#1} \langle #2 \mid #3 \rangle}

\newcommand{\lra}[1]{\xleftrightarrow{\ast}_{#1}}
\newcommand{\xra}[1]{\xrightarrow{}^\ast_{#1}}
\newcommand{\xr}[1]{\xrightarrow{}_{#1}}

\newcommand{\trev}{\text{rev}}

\newcommand{\Z}{\mathbb{Z}}
\newcommand{\ZM}{\mathbb{Z}M}

\newcommand{\FIS}{{\mathbf{F}_1^\star}}
\newcommand{\FrS}{\mathbf{F}_r^\star}

\newcommand{\cL}{\mathcal{L}}

\newcommand{\cR}{\mathcal{R}}
\newcommand{\cE}{\mathcal{E}}

\DeclareMathOperator{\Irr}{Irr}

\DeclareMathOperator{\FP}{FP}
\DeclareMathOperator{\Ext}{Ext}
\DeclareMathOperator{\Hom}{Hom}
\DeclareMathOperator{\Tor}{Tor}
\DeclareMathOperator{\FSX}{\mathbf{F}^\star}

\DeclareMathOperator{\im}{im}
\DeclareMathOperator{\FIM}{FIM}

\newcommand{\N}{{\mathbb{N}}}

\newtheorem{innercustomgeneric}{\customgenericname}
\providecommand{\customgenericname}{}
\newcommand{\newcustomtheorem}[2]{%
  \newenvironment{#1}[1]
  {%
   \renewcommand\customgenericname{#2}%
   \renewcommand\theinnercustomgeneric{##1}%
   \innercustomgeneric
  }
  {\endinnercustomgeneric}
}
\newtheorem{theorem}{Theorem} 
\newtheorem*{theorem*}{Theorem} 
\newcustomtheorem{customthm}{Theorem}
\newcustomtheorem{customcor}{Corollary}
\newcustomtheorem{customprop}{Proposition}

\newtheorem{lemma}{Lemma}     
\newtheorem{corollary}{Corollary}
\newtheorem{proposition}{Proposition}
\numberwithin{lemma}{section}
\numberwithin{corollary}{section}
\numberwithin{proposition}{section}

\newtheorem*{mainlemma*}{Main Lemma}

\theoremstyle{definition}

\newtheorem{question}{Question}
\newtheorem*{question*}{Question}

\numberwithin{example}{section}
\newtheorem{remark}{Remark}

\begin{document}

\title[]{On the growth and integral (co)homology of free regular star-monoids}

\author{\small{Carl-Fredrik Nyberg-Brodda}}
\address{{School of Mathematics, Korea Institute for Advanced Study (KIAS), Seoul 02455, Korea}}
\email{cfnb@kias.re.kr}
\thanks{The author is supported by the Mid-Career Researcher Program (RS-2023-00278510) through the National Research Foundation funded by the government of Korea, and by the KIAS Individual Grant MG094701 at Korea Institute for Advanced Study}

\subjclass[2020]{}

\date{\today}


\begin{abstract}
The free regular $\star$-monoid of rank $r$ is the freest $r$-generated regular monoid $\FrS$ in which every element $m$ has a distinguished pseudo-inverse $m^\star$ satisfying $mm^\star m = m$ and $(m^\star)^\star = m$. We study the growth rate of the monogenic regular $\star$-monoid $\FIS$, and prove that this growth rate is intermediate. In particular, we deduce that $\FrS$ is not rational or automatic for any $r \geq 1$, yielding the analogue of a result of Cutting~\& ~Solomon for free inverse monoids. Next, for all ranks $r \geq 1$ we determine the integral homology groups $H_\ast(\FrS, \Z)$, and by constructing a collapsing scheme prove that they vanish in dimension $3$ and above. As a corollary, we deduce that the free regular $\star$-monoid $\FrS$ of rank $r \geq 1$ does not have the homological finiteness property $\FP_2$, yielding the analogue of a result of Gray \& Steinberg for free inverse monoids. 
\end{abstract}

\maketitle

\noindent Regular $\star$-monoids were first studied by Nordahl \& Scheiblich \cite{Nordahl1978}, and have recently resurfaced in the study of generalizations of diagram monoids \cite{East2024, EastGray2024}. In particular, in \cite{EastGray2024} the notion of \textit{free projection-generated} regular $\star$-monoids was studied. A natural object that arises in this context is the (absolutely) free regular $\star$-monoid $\FSX(X)$ on a set, which is the main object of study in this article. In particular, we will study its rate of growth and its homological properties. The monogenic free regular $\star$-monoid, which we denote $\FIS$, has some peculiar properties. for example, it is orthodox \cite{Yamada1983}, though it is not inverse. The first main result of this article is to show that $\FIS$ has intermediate growth (Theorem~\ref{Thm:main-thm-growth}). We prove this by connecting the problem of determining its growth to a known counting problem, which is connected to rather deep number-theoretic problems. In \S\ref{Sec:homology-of-FRS} we prove our second main result (Theorem~\ref{Thm:main-homology-theorem}) and compute all integral homology groups $H_k(\FrS, \Z)$ of $\FrS$ for all ranks $r \geq 1$. In particular, we prove that the trivial cohomological dimension of $\FrS$ is $2$, and that $\FrS$ does not have the homological finiteness property $\FP_2$. 

\section{Preliminaries}

\subsection{Regular $\star$-monoids}

A regular monoid $M$ is one in which every element $m \in M$ has a pseudo-inverse $m'$, i.e.\ an element satisfying $mm'm = m$. Free regular monoids do not exist (cf.\ \cite[Corollary~3.2]{McAlister1968}), but if one augments the signature one can construct the variety of \textit{regular $\star$-monoids}, which has free objects. These will be the main focus of this article; we now construct them. A $\star$-monoid is a monoid $M$ equipped with an involution $\star \colon M \to M$ such that $(xy)^\star = y^\star x^\star$. All groups and inverse monoids are, of course, $\star$-monoids. Between the class of all $\star$-monoids and that of all inverse monoids lies the class of \textit{regular $\star$-monoids}, being those $\star$-monoids for which $m^\star$ is a pseudo-inverse of $m$, i.e.\ $mm^\star m = m$. Notice that from this relation it also follows that $m$ is a pseudo-inverse of $m^\star$, as $m^\star m m^\star = ((m^\star m m^\star)^\star)^\star = (m m^\star m)^\star = m^\star$. The variety of regular $\star$-monoids has free objects; thus \textit{free regular $\star$-monoids exist}.

Due to the combinatorial nature of this present article, we opt to give an equivalent definition of free regular $\star$-monoids via monoid presentations, by specifying generators and relations, rather than via universal algebra. Let $A$ be a finite set (an alphabet), and let $A^\ast$ denote the free monoid on $A$. The empty word will be denoted $1$. Let $A_\star$ denote a set with $A \cap A_\star = \varnothing$ and in involutive correspondence with $A$, associating $a$ with $a^\star$, and extend this correspondence to words such that $(xy)^\star = y^\star x^\star$ for all words $x, y \in A^\ast$. Then the free regular $\star$-monoid $\FSX(A)$ is the monoid with the presentation 
\begin{equation}\label{Eq:Presentation-of-freestar}
\FSX(A) = \pres{Mon}{A \cup A_\star}{ww^\star w = w, \: (\forall w \in (A \cup A_\star)^\ast)}.
\end{equation}
The monoid $\FSX(A)$ will be studied in this article, using \eqref{Eq:Presentation-of-freestar} as a definition; for a proof that $\FSX(A)$ is indeed a free regular $\star$-monoid, see \cite[Theorem~5]{Polak2001}. When $|A| = r$, we will write $\FrS$ for the $r$-generated free regular $\star$-monoid. Note that considered as a regular $\star$-monoid $\FrS$ is generated by $r$ elements, whereas considered as a monoid it is generated by $2r$ elements. 

\subsection{Growth}\label{Subsec:growth}

If $M$ is any monoid generated by a finite set $A$, then the \textit{length} of an element $m \in M$ is the least $k$ such that $m$ can be written as a word of length $k$ over $A$. The \textit{ball of radius $k$} is the set of all elements in $M$ with length exactly $k$. The \textit{growth function} of $M$ (with respect to $A$) is then the function $\gamma_A \colon \mathbb{N} \to \mathbb{N}$ which maps $k$ to the number of elements in the ball of radius $k$. The \textit{growth rate} of $M$ is defined as $\limsup_{n \to \infty} \gamma(n)^{\frac{1}{n}}$. This limit always exists, since the growth rate is submultiplicative; if the limit is greater than $1$, then $M$ is said to have \textit{exponential growth}. If there exist natural numbers $C, d \in \N$ such that $\gamma(n) \leq (Cn)^d$ for all $n \in \N$, then we say that $M$ has \textit{polynomial growth} (of degree $\leq d$). If $M$ has neither polynomial nor exponential growth, then we say that $M$ has \textit{intermediate growth}. For more on the basics of growth, we refer the reader to the survey by Grigorchuk \cite{Grigorchuk1991}.

\subsection{Rewriting systems}\label{Subsec:rewriting-homology}

A \textit{rewriting system} $\mathcal{R}$ on $A$ is a subset of $A^\ast \times A^\ast$, and its elements $(u, v)$ are called \textit{rules}, which we will write as $(u \longrightarrow v) \in \mathcal{R}$. The system $\mathcal{R}$ induces several relations on $A^\ast$. We write $u \xr{\mathcal{R}} v$ if there exist $x, y \in A^\ast$ and a rule $(\ell \longrightarrow r) \in \mathcal{R}$ such that $u \equiv x\ell y$ and $v \equiv xry$. Any word $u$ for which there exists some $v$ such that $u \xr{\mathcal{R}} v$ is called $(\mathcal{R}$-)\textit{reducible}; otherwise, $u$ is ($\mathcal{R}$-)\textit{irreducible}. Let $\xra{R}$ denote the reflexive and transitive closure of $\xr{R}$. The transitive, symmetric closure of $\xra{\mathcal{R}}$ is a congruence, denoted $\lra{\mathcal{R}}$. The monoid $A^\ast / \lra{\mathcal{R}}$ is called the monoid defined by $\mathcal{R}$, and is clearly just the monoid generated by $A$ and subject to the defining relations $u = v$ for all $(u, v) \in \mathcal{R}$. If there is no infinite chain $u_0 \xr{\mathcal{R}} u_1 \xr{\mathcal{R}} \cdots$, then we say $\mathcal{R}$ is \textit{terminating}. We say that $\mathcal{R}$ is \textit{confluent} if for all words $u, v_1, v_2 \in A^\ast$ such that $u \xra{\mathcal{R}} v_1$ and $u \xra{\mathcal{R}} v_2$, there exists some word $v_3$ such that $v_1, v_2 \xra{\mathcal{R}} v_3$. A confluent and terminating system is said to be \textit{complete}. Clearly, in a complete system any word $u$ can be rewritten to a unique normal form: the unique irreducible word $v$ such that $u \xra{\mathcal{R}} v$. For more details on string rewriting, we refer the reader to \cite{Jantzen1988}. 

Any monoid presentation is a string rewriting system, although the resulting system is by no means guaranteed to be complete. Starting with the presentation \eqref{Eq:Presentation-of-freestar}, we obtain the rewriting system
\begin{equation}\label{Eq:free-star-rewriting}
\mathbf{R}^\star(A) = \{ ww^\star w \longrightarrow w, \: (\forall w \in (A \cup A_\star)^\ast) \}
\end{equation}
By analogy with the notation $\FrS$, we will write $\mathbf{R}^\star_r$ for $\mathbf{R}^\star(A)$ if $|A| = r$ and the alphabet is implictly understood. The rewriting system \eqref{Eq:free-star-rewriting} is infinite, but it is, of course, computable; and one of the fundamental results that serve as a starting point for this article is the fact that the system yields a normal form for elements of $\FSX(A)$. 

\begin{proposition}[Pol\'ak, 2001]\label{Prop:Polak-prop}
The rewriting system $\mathbf{R}^\star(A)$ is complete.
\end{proposition}

In the particular case of $r=1$, an even simpler rewriting system for $\FrS$ can be constructed. Indeed, let $A = \{ a \}$, and let
\begin{equation}\label{Eq:R1-definition}
\mathcal{R}_1 = \{ a^n (a^\star)^n a^n \longrightarrow a^n, \:\: (a^\star)^n a^n (a^\star)^n \longrightarrow (a^\star)^n \mid n \geq 1 \}.
\end{equation}
It is clear and easy to prove that $\mathcal{R}_1$ is complete. Furthermore, a simple combinatorial argument shows that $\mathcal{R}_1$ is, in fact, equivalent to $\mathbf{R}^\star_1$; we include the proof for completeness.

\begin{lemma}\label{Lem:aiAiai-is-equivalent-to-full}
The rewriting systems $\mathbf{R}^\star_1$ and $\mathcal{R}_1$ are equivalent. 
\end{lemma}
\begin{proof}
Obviously, $\mathcal{R}_1 \subseteq \mathbf{R}^\star_1$, so it suffices to show that any left-hand side in $\mathbf{R}^\star_1$ can be reduced to its corresponding right-hand side under $\mathcal{R}_1$. For ease of notation, let $A$ denote $a^\star$, and let $w \in \{ a, A \}^\ast$ be any $\mathcal{R}_1$-irreducible word. Let Then $w$ does not contain any subword of the form $a^n A^n a^n$ or $A^n a^n A^n$ for any $n \geq 1$, and can therefore be uniquely written as a word $w_0 w_1 w_2$, where, when written as alternating words over $a, A$, the exponents of the terms in $w_0$ are strictly increasing; the exponents of the terms in $w_2$ are strictly decreasing, and $w_1 \in \{ a^m, A^m, a^m A^m \}$ with $m$ larger than any exponent in $w_0$ and $w_2$. In particular, either $w = w_1$, in which case it is easy to see that $ww^\star w$ is $\mathcal{R}_1$-reducible, and the claim would follow; or else $w$ begins with $x^i (x^\star)^j$ with $i<j$ and $x$ is a letter; or else $w$ ends with $x^i (x^\star)^j$ with $i>j$ and $x$ is a letter. In the second case, we have that $w^\star w$ contains $x^j (x^\star)^i x^i (x^\star)^j$, and since $i < j$ this word contains $(x^\star)^i x^i (x^\star)^i$. Hence, $ww^\star w$ is $\mathcal{R}_1$-reducible in this case. In the last case, an analogous argument on $ww^\star$ also yields that $ww^\star w$ is $\mathcal{R}_1$-reducible. Thus, in all cases we are done by induction on the length of $w$. 
\end{proof}

Thus, $\mathcal{R}_1$ is a complete rewriting system which defines $\FIS$. The rest of the article will use this rewriting system to derive properties about $\FIS$.

\subsection{Monoid homology and the Kobayashi resolution}\label{Subsec:kobayashi}

Let $M$ be any monoid. Given any $\ZM$-module $A$ and a projective resolution $(P_\ast, \partial_\ast)$ of $\Z$ as a trivial right $\ZM$-module, we define the \textit{right homology} of $M$ with coefficients in $A$ as $H_\ast^{(r)}(M, A) = \Tor_\ast^{\ZM}(\Z, A)$, and the \textit{right cohomology} of $M$ with coefficients in $A$ as $H^\ast_{(r)}(M, A) = \Ext^\ast_{\ZM}(\Z, A)$. We analogously define $H_\ast^{(l)}(M, A)$ and $H^\ast_{(l)}(M, A)$ as the \textit{left} homology (resp.\ cohomology) groups. In general, the left and right versions of these groups are independent of one another (cf.\ e.g.\ \cite{GubaPride1998}), but for monoids equipped with an anti-isomorphism ${}^\circ \colon M \to M$, i.e.\ such that such that $(mn)^\circ = n^\circ m^\circ$ and $(m^\circ)^\circ = m$, then the left and right (co)homology groups can be shown to be identical; this is, for example, the case when $M$ is a group, an inverse monoid, or a regular $\star$-monoid. In these cases, we can simply speak of the \textit{(co)homology groups} of $M$ with coefficients in $A$. Furthermore, we say that $M$ is type (left) \textit{$\FP_n$} if there is a projective $\ZM$-resolution $(P_\ast, \partial_\ast)$ of $\Z$ (as a left $\ZM$-module) in which $P_i$ is finitely generated for all $1 \leq i \leq n$. Analogously, we may define \textit{right} $\FP_n$, which is again equivalent to left $\FP_n$ when $M$ is equipped with an anti-isomorphism as above. Thus, for the remainder of the article we will only discuss the properties $\FP_n$. If $M$ is of type $\FP_n$ for all $n \geq 1$, then we say it is type $\FP_\infty$. 

From a complete rewriting system $\mathcal{R}$ defining a monoid $M$, Kobayashi \cite{Kobayashi1990} constructed a free $\Z M$-resolution $(P_\ast, \partial_\ast)$ of $\Z$ with the property that if $\mathcal{R}$ is finite, then every $P_i$ is finitely generated; thereby proving the remarkable result that any monoid admitting a finite complete rewriting system is $\FP_\infty$ (see also \cite{Anick1986, Groves1990}). We will follow the somewhat simplified notation of Guba \& Pride \cite{GubaPride1996} for the Kobayashi resolution. We will only construct a resolution of $\Z$ as a \textit{right} $\ZM$-module; the interested reader may find the analogous left resolution by reading this article in a mirror.

Let $A$ be an alphabet, and $\cR \subseteq A^\ast \times A^\ast$ be a complete rewriting system presenting a monoid $M$. Let $\Irr(\cR)$ be the set of all $\cR$-irreducible words (henceforth simply irreducible words), and let $\Irr^+(\cR) = \Irr(\cR) \setminus \{ 1 \}$. We assume without loss of generality that no element of $A$ is a left-hand side in $\cR$. The irreducible  form of a word $w \in A^\ast$ will be denoted by $\overline{w}$. Since $\cR$ is complete, we identify $\Irr(\cR)$ with $M$, and speak of the elements of $M$ as irreducible words. Let $\cE$ be the set of all pairs of words $(u, v)$ with $u, v \in \Irr^+(\cR)$, such that $uv$ is \textbf{not} irreducible but all of its proper prefixes are irreducible. For $n \geq 1$, let $V^{(n)}$ be the set of all sequences of the form $(v_1, v_2, \dots, v_n)$ such that $v_1 \in A$, and for all $1 \leq i < n$ we have $(v_i, v_{i+1}) \in \cE$. Let $P_n$ be the free right $\ZM$-module on the basis $V^{(n)}$. Setting $P_{-1} = \Z$ and $P_0 = \ZM$, we will now define chain maps $\partial_\ast$ such that $(P_\ast, \partial_\ast)$ is a free resolution, which we will call the \textit{Kobayashi resolution} associated to $\cR$. We remark that if $\cR$ is finite then each $P_i$ is finitely generated. 

We construct the chain maps $\partial_\ast$ inductively, at the same time as defining a $\Z$-homomorphism $i_\ast$ such that it is a chain homotopy, i.e.\ $\partial i  + i \partial = \operatorname{id}$, thereby ensuring $(P_\ast, \partial_\ast)$ is indeed a resolution. We set $\partial_0 = \varepsilon$, the augmentation map, and let $i_0 \colon \Z \to \ZM$ be defined by $i_0(1) = 1$. Let $\partial_1 \colon P_1 \to \ZM$ be defined by $\partial_1(a) = a-1$ for $a \in A = V^{(1)}$, and extended $\ZM$-linearly. We define $i_1 \colon \ZM \to P_1$ as follows. Let $x \in \Irr(\cR)$, and write $x = x_1 x_2 \cdots x_k$ with $x_i \in A$. Then we set $i_1(x) = \sum_{i=1}^k (x_i) \circ x_{i+1} \cdots x_k$ where the notation $\circ$ is used to indicate the right action of $\ZM$ on $P_1$ (and, more generally, $P_i$). In particular, we set $i_1(1) = 0$. 

We begin the inductive definition of the remaining $i_\ast$ and $\partial_\ast$. First, to simplify notation, we will write $\partial_n( (v_1, \dots, v_n) \circ 1)$ as $\partial_n(v_1, \dots, v_n)$; by $\ZM$-linearity of the constructed map this notation will be sufficient. For $n \geq 1$, we set
\begin{equation}\label{Eq:partial_n-definition}
\partial_{n+1}(v_1, \dots, v_n, v_{n+1}) = (v_1, \dots, v_n) \circ v_{n+1} - i_n \left( \partial_n (v_1, \dots, v_n) \circ v_{n+1} \right).
\end{equation}
To define $i_{n+1}$, notice that since $P_n$ is free as a $\Z$-module on the set of elements $(v_1, \dots, v_n) \circ x$, where $(v_1, \dots, v_n) \in V^{(n)}$ and $x \in \Irr(\cR)$, we only need to define $i_{n+1}$ on this set. Fix such an element. If $v_nx$ is irreducible, then we set $i_{n+1}( (v_1, \dots, v_n) \circ x) = 0$. Otherwise, there is some minimal prefix $v_n v_{n+1}$ of $v_n x$ which is reducible, where $v_{n+1}$ is non-empty; write $x = v_{n+1} y$. Then $(v_n, v_{n+1}) \in \cE$, and by \cite[p.~266]{Kobayashi1990} the element $i_n \left( \partial_n( v_1, \dots, v_n) \circ v_{n+1} \right)$ can be written as a sum of basis elements of the form $(w_1, \dots, w_n) \circ z$, for all of which $w_1 \cdots w_n z$ is shorter than $v_1 \cdots v_n v_{n+1}$. In particular, $i_{n+1}$ is defined on such basis elements by induction (on word length), and hence also on $i_n \left( \partial_n ( (v_1, \dots, v_n) \circ v_{n+1} )  \circ y \right)$. We then define 
\begin{equation}\label{Eq:i-n-definition}
i_{n+1}( (v_1, \dots, v_n) \circ x) = (v_1, \dots, v_n, v_{n+1}) \circ y + i_{n+1} \left( i_n \left( \partial_{n}(v_1, \dots, v_n) \circ v_{n+1} \right) \circ y \right)
\end{equation}
which completes the inductive definition of $i_\ast$ and $\partial_\ast$. It can be shown that $i_\ast$ is a chain homotopy, and hence the Kobayashi resolution is acyclic. 

\subsection{Homology and collapsing schemes}\label{Subsec:collapsing-schemes}

The complexes built using Kobayashi's method are often quite unwieldy; to simplify computing their homology, we will use \textit{collapsing schemes} in the sense of Brown \cite{Brown1992} and of Forman's discrete Morse theory \cite{Forman1998}. We will follow, in part, the exposition of this theory given by Farley \cite{Farley2005}. For any commutative unital ring $R$, let $(C_\ast, \partial_\ast)$ be a chain complex of free $R$-modules. Fix a basis $B_n$ for each $C_n$, and let $\overline{B}_n = - B_n \cup B_n \cup \{ 0 \}$. A map $\theta \colon B_n \to \overline{B}_{n+1}$ is said to be a \textit{collapsing scheme} (or \textit{Morse matching}) if:
\begin{enumerate}
\item $\theta^2 = 0$, 
\item for all $b, b' \in B_n$, if $\theta(b) = \pm \theta(b') \neq 0$, then $b = b'$,
\item for all $b \in B_n$, if $\theta(b) \neq 0$ then $b$ occurs in $(\partial \theta)(b)$ with the coefficient $-1$;
\item for all distinct $b, b' \in B_n$, we write $b \succ b'$ if $b'$ occurs in $(\partial \theta)(b)$ or $(\theta \partial)(b')$ with a non-zero coefficient; letting $\succ$ also denote its transitive closure, we require that the poset $(B_n, \succ)$ be locally finite and that $\succ$ is well-founded.
\end{enumerate}

If $b \in B_n$, then we say that $b$ is
\begin{itemize}
\item \textit{critical}, if $\theta(b) = 0$ and $b \not\in \im \theta$; 
\item \textit{collapsible}, if $b \in \im \theta$;
\item \textit{redundant}, if $\theta(b) \neq 0$. 
\end{itemize}
Clearly, a basis element must lie in exactly one of the above three classes, and we thus have a partition of each $B_n$. The goal of any collapsing scheme is to minimize the number of critical cells; the fewer, the simpler all homological computations for $(C_\ast, \partial_\ast)$ will become, by virtue of the central theorem of collapsing scheme. To state it, let 
\[
\Pi_{\theta} := 1 + \partial \theta + \theta \partial.
\]
Then it can be shown that $\Pi_\theta \colon C_\ast \to C_\ast$ is homotopic to $\operatorname{id}$, and furthermore that for any $c \in C_\ast$, there exist some $n \in \N$ such that $\Pi_\theta^n(c) = \Pi_\theta^{n+1}(c)$ (see \cite[Proposition~16]{Farley2005}). This induces a well-defined map $\Pi_\theta^\infty \colon C_\ast \to C_\ast$. Let $I_\theta(C_\ast) = \Pi_\theta^\infty(C_\ast)$ be the $R$-module of $\Pi_\theta$-invariant chains in $C_\ast$. This module can be described explicitly; indeed, let $C^\theta_n$ be the free $R$-submodule of $C_n$ generated by the critical basis elements, and let $\pi_\theta \colon C_n \to C^\theta_n$ denote the projection. Then one can prove that $I_\theta(C_\ast) \cong C^\theta_n$ under $\Pi_\theta^\infty$ (see \cite[Proposition~17]{Farley2005}). If we identify $I_\theta(C_\ast)$ and $C^\theta_\ast$ using this isomorphism, then we can define a new chain map $\partial^\theta_\ast \colon C^\theta_\ast \to C^\theta_\ast$ by $\partial^\theta_\ast = \pi^\theta \circ \partial \circ \Pi^\infty_\theta$. The complex $(C_\ast^\theta, \partial^\theta_\ast)$ is called the \textit{Morse complex} associated with $\theta$. 

\begin{proposition}[Forman \cite{Forman1998}]\label{Prop:Forman'sTheorem}
There is an isomorphism $H_\ast(C^\theta_\ast, \partial^\theta_\ast) \cong H_\ast(C_\ast, \partial_\ast)$. 
\end{proposition}

Note, for example, that if $b \in C_\ast$ is collapsible and minimal under $\succ$, then $\Pi_\theta(b) = 0$, and in general one proves the above theorem by induction on $\succ$ that any collapsible or redundant cell is mapped to $0$ by $\Pi_\theta^\infty$. See also \cite{Brown1992} for a geometric version, and \cite[Theorem~18]{Farley2005} for a lucid algebraic proof.

\section{Growth of the free monogenic regular $\star$-monoid}

\noindent In this section, we will prove that the monogenic regular $\star$-monoid $\FIS$ has intermediate growth. To do this, we will use the rewriting system $\mathcal{R}_1$ as in \eqref{Eq:R1-definition}, and in particular determine the growth type of the language of irreducible elements modulo $\mathcal{R}_1$. Since the system $\mathcal{R}_1$ is length-reducing, each such irreducible element is the shortest representative in its equivalence class, and hence the growth rate of this language is precisely the growth rate of $\FIS$ (with respect to the generating set $\{ a, a^\star \}$. 

To ease notation, we will consider binary strings over $\{ 0, 1 \}$, with $a$ corresponding to $0$ and $a^\star$ to $1$. Consider the following two languages: 
\begin{align*}
\cL_\trev &= \{ w w^\trev w \mid w \in \{ 0, 1\}^\ast \}, \\
\cL_\star &= \{ w w^\star w \mid w \in \{ 0, 1\}^\ast \},
\end{align*}
where ${}^{\star}$ denotes bitwise negation followed by reversal (as in $\FIS$). Thus, for example,
\[
(011)(110)(011) \in \cL_\trev, \quad \text{and} \quad (011)(001)(011) \in \cL_\star,
\]
but neither of the two words are elements of the other set. Let $\cL_\trev^\complement(n)$ resp.\ $\cL_\star^\complement(n)$ denote the set of all binary strings of length at most $n$ not containing any element of $\cL_\trev$ resp.\ $\cL_\star$. Thus, for example, $\cL_\trev^\complement(n) = \{ 0, 1\}^{\leq n} \setminus \{ 0, 1\}^\ast \cL_\trev(n) \{ 0, 1\}^\ast$, where $\{ 0, 1\}^{\leq n}$ denotes the set of all binary strings of length at most $n$.

The language $\cL_\trev$ has been studied in the literature (and variants of it, e.g.\ replacing $ww^\trev w$ by $www^\trev$ in the definition \cite{CurrieRampersad2016}), and a good deal of information about the rate of growth of its complement has been obtained. 

\begin{proposition}[Currie \& Rampersad \cite{CurrieRampersad2015}]\label{Prop:growth-of-c-trev}
The growth rate of $\cL_\trev^\complement(n)$ is intermediate.
\end{proposition}

We remark that the method of proof goes via connecting the growth of $\cL_\trev^\complement(n)$ with \textit{strongly unimodal sequences}, whose rate of growth $\widetilde{u}(n)$ is known to be intermediate by a result of Rhoades \cite{Rhoades2014}. Indeed, by passing via Ramanujan mock theta functions, Rhoades found the asymptotics
\[
\widetilde{u}(n) = \frac{\sqrt{3}}{(24n-1)^{3/4}} \exp \left( \frac{\pi}{6} \sqrt{24n-1}\right) \left( 1 + \frac{\pi^2-9}{4\pi (24n-1)^{1/2}} + O\left( \frac{1}{n} \right) \right).
\]
The authors of \cite{CurrieRampersad2015} prove that $\widetilde{u}(n) \leq |\cL_\trev^\complement(n)| \leq 2\widetilde{u}(n)$, which thus yields their result. As we shall see in Theorem~\ref{Thm:main-thm-growth}, these bounds also apply to the growth rate of $\FIS$. 

Unlike $\cL_\trev$, the language $\cL_\star$ and its complement have not, as far as we are aware, been studied in the literature. Its relevance to us comes from the following simple result.

\begin{lemma}\label{Lem:radius-is-just-L-1}
The ball of radius $n$ in $\FIS$ has size $|\cL_\star^\complement(n)|$. 
\end{lemma}
\begin{proof}
Recall that the rewriting system $\mathbf{R}^\star_1$, defined in \eqref{Eq:free-star-rewriting}, defines $\FIS$ and is complete by Proposition~\ref{Prop:Polak-prop}. Since $\mathbf{R}^\star_1$ is also length-reducing, the shortest representative of any word is precisely its $\mathbf{R}^\star_1$-irreducible descendant; and hence the ball of radius $n$ in $\FIS$ has the same size as the set of words over $\{ a, a^\star \}$ of length $n$ not containing any left-hand side of $\mathbf{R}^\star_1$, which is precisely $|\cL_\star^\complement(n)|$. 
\end{proof}

As noted above, $\cL_\trev \neq \cL_\star$. For an infinite word $Q  \in \{ 0, 1\}^\N$, and for $v \in \{ 0, 1 \}^\ast$ let $\Phi_Q(v)$ be defined by $\Phi_Q(v) = v \oplus Q_{|v|}$, where $\oplus$ denotes the nimsum, i.e.\ bitwise binary addition, and $Q_{n}$ denotes the prefix of length $n$ of $Q$. Clearly, the map $\Phi_Q$ is a bijection -- indeed, it is its own inverse -- and length-preserving. Remarkably, for some words $Q$ the map $\Phi_Q$ also restricts to a bijection between the languages $\cL_\trev$ and $\cL_\star$. Let $\omega = (01)^\infty$, and let $\omega^{-1} = (10)^\infty = \omega \oplus 1^\infty$. Then it is clear that 
\begin{equation}\label{Eq:(01)-mult-property}
\Phi_\omega(uv) = \begin{cases*} 
\Phi_\omega(u) \Phi_\omega(v) & if $|u|$ is even, \\
\Phi_\omega(u) \Phi_{\omega^{-1}}(v) & if $|u|$ is odd. 
\end{cases*}
\end{equation}
Note that for any word $u$, if we set $n = |u|$ then $u^\star = (u \oplus 1^{n})^\trev$. Furthermore, we have 
\begin{equation}\label{Eq:1-prop}
\Phi_\omega(1^n) = 
\begin{cases*}
\omega_n^\trev & if $n$ is even, \\
\omega^{-1}_n & if $n$ is odd.
\end{cases*}
\end{equation}
Using these elementary properties, we can now prove that $\Phi_\omega$ is a bijection between our two languages. First, as a simple example, if we have the word $u = 011$, then $uu^\trev u = (011)(110)(011)$. Applying $\Phi_\omega$ to $uu^\trev u$ then yields the word $(001)(011)(001)$, which is of the form $vv^\star v$ for $v = 001$. We now prove that this always works. 

\begin{proposition}\label{Prop:exists-bijection-between-excluding}
The map $\Phi_\omega$ maps any word of the form $uu^{\operatorname{rev}} u$ to a word of the form $vv^\star v$. In particular, for all $n \geq 1$, $\Phi_\omega$ is a bijection between $\cL_\star^\complement(n)$ and $\cL_\trev^\complement(n)$.
\end{proposition}
\begin{proof}
Since $\Phi_\omega$ is length-preserving and a bijection (on the set of all words), the second part of the proposition follows from the first part. Thus, let $u \in \{  0, 1 \}^\ast$ be any word. We will prove that 
\begin{equation}\label{Eq:key-property}
\Phi_\omega(u u^\trev u) = vv^\star v, \quad \text{ where } v = \Phi_\omega(u),
\end{equation}
which will establish the claim. Throughout, let $|u| = n$. To prove \eqref{Eq:key-property}, suppose first that $n$ is even. Then we have 
\begin{align*}
\Phi_\omega(u^\trev) = u^\trev \oplus \omega_{n} = ( u \oplus \omega_n^\trev)^\trev = (u \oplus \omega_n \oplus 1^n)^\trev = (u \oplus \omega_n)^\star = \Phi_\omega(u)^\star
\end{align*}
where the third equality follows from the fact that $n$ is even. Thus, by \eqref{Eq:(01)-mult-property}, we have 
\begin{align*}
\Phi_\omega(uu^\trev u) = v \Phi_\omega(u^\trev) v =  v v^\star v
\end{align*}
thus establishing \eqref{Eq:key-property} as required. Suppose now instead that $n$, the length of $u$, is odd. Then by \eqref{Eq:(01)-mult-property} we have
\begin{align*}
\Phi_\omega(uu^\trev u) = \Phi_\omega(uu^\trev) \Phi_\omega(u) = \Phi_\omega(u) \Phi_{\omega^{-1}}(u^\trev) \Phi_\omega(u)
\end{align*}
as the length of $uu^\trev$ is even. Thus, to establish \eqref{Eq:key-property}, we now only need to establish $\Phi_{\omega^{-1}}(u^\trev) = v^\star$. But this follows directly in a similar manner as above, since
\begin{align*}
\Phi_{\omega^{-1}}(u^\trev) = u^\trev \oplus \omega_n^{-1} &= u^\trev \oplus \omega_n \oplus 1^n \\
&= (u \oplus (\omega_n \oplus 1^n)^\trev)^\trev \\
&= (u \oplus (\omega_n)^\star)^\trev \\
&= (u \oplus \omega^{-1}_n)^\trev \\
&= (u \oplus \omega_n \oplus 1^n)^\trev = (u \oplus \omega_n)^\star = \Phi_\omega(u)^\star = v^\star
\end{align*}
where the second equality follows from the fact that $n$ is odd. This completes the proof of \eqref{Eq:key-property} in the case that $n$ is odd, and hence also completes the proof of the proposition.
\end{proof}

Thus, combining Lemma~\ref{Lem:radius-is-just-L-1} with Proposition~\ref{Prop:exists-bijection-between-excluding}, we see that the ball of radius $n$ in $\FIS$ has size $|\cL_\trev^\complement(n)|$. Hence, by Proposition~\ref{Prop:growth-of-c-trev}, we have proved the main theorem of this section. 

\begin{theorem}\label{Thm:main-thm-growth}
The monogenic free regular $\star$-monoid $\FIS$ has intermediate growth. 
\end{theorem}

Recall that a monoid is said to be \textit{rational} if there is regular language of unique normal forms for the monoid; and that a monoid is \textit{automatic} if there is a regular language of normal forms such that checking equality or whether two words differ by a factor of a single generator can be done by a finite state automaton. We will not use the precise definition in this article, instead referring the reader to \cite{Campbell2001}. Indeed, the only property of automatic monoids we will use is the fact that they are rational, i.e.\ admit a regular language of normal forms \textit{with uniqueness} \cite[Corollary~5.6]{Campbell2001}. In \cite{CuttingSolomon2001}, the authors prove that, unlike free groups and free monoids, free \textit{inverse} monoids are neither rational nor automatic. While their method is quite different from ours (since free inverse monoids never have intermediate growth \cite{CFGrowth}), our conclusions for free regular $\star$-monoids are analogous:

\begin{corollary}
No free regular $\star$-monoid is rational. Therefore, no free regular $\star$-monoid is automatic. 
\end{corollary}
\begin{proof}
If $\FIS$ were rational, then its growth function would be alternative, i.e.\ exponential or polynomial, since the generating function for the number of words of a fixed length would be a rational function. Thus, by Theorem~\ref{Thm:main-thm-growth} the monogenic free regular $\star$-monoid $\FIS$ is not rational (and therefore not automatic). Let $X = \{ x_1, x_2, \dots \}$ and consider the free regular $\star$-monoid $\FSX(X)$. Let $I_1 \subset \FSX(X)$ be the set of all elements which can be represented by a word not involving the letter $x_1$. Then since $\mathbf{R}^\star(X)$ defines $\FSX(X)$ and for each rule in $\mathbf{R}^\star(X)$ the set of letters in the left-hand side is the same as the set of letters in the right-hand side, it follows that $I_1$ is a ideal of $\FSX(X)$. The Rees quotient monoid $\FSX(X) / I_1$ is $\FIS$ with a zero adjoined. Since adjoining a zero to a rational monoid results in a rational monoid (\cite[Proposition~1.9]{CuttingSolomon2001}), and since a Rees quotient with no zero divisors of a monoid $M$ is rational if $M$ is rational (\cite[Theorem~1.10]{CuttingSolomon2001}), it follows that $\FSX(X)$ is not rational; and therefore also not automatic. 
\end{proof}

Returning briefly to languages, we note that one of the consequences of Proposition~\ref{Prop:exists-bijection-between-excluding} is that $\varphi(\cL_\trev) = \cL_\star$. Another consequence, in view of Lemma~\ref{Lem:aiAiai-is-equivalent-to-full}, of independent interest is the following. For $n \geq 1$, let $U_n = 0^n 1^n 0^n$ and $V_n = 1^n 0^n 1^n$. Then we have the following characterization:

\begin{corollary}
A minimal set of forbidden subwords for the set $\cL_\trev^\complement$ of words avoiding 
\[
\{ w w^\trev w \mid w \in \{ 0, 1 \}^\star \}
\]
is the set of words $\varphi( \{ U_n, V_n \mid n \geq 1 \})$, i.e.\
\[
\{ 000, 111, 011001, 100110, 010010010, 101101101, \dots \}
\]
\end{corollary}

Thus, we only need to exclude at most two words of each length, which is a significantly smaller set of excluded subwords than simply excluding $w w^\trev w$ for all words $w$, as this is a list of  $~\sim 2^n$ excluded subwords of length $\leq n$.

We do not know how to determine the rate of growth of $\FrS$ for $r \geq 2$. It is, of course, exponential (since $\FrS$ surjects the free group on two generators), but we suspect that this exponential growth rate is not algebraic, by analogy with the rather complicated number-theoretic links with the intermediate growth of $\cL_\trev^\complement$ in Proposition~\ref{Prop:growth-of-c-trev}. 

\begin{question}
What is the exponential growth rate of $\FrS$? Is it an algebraic number?
\end{question}

The rate of growth of free inverse monoids $\FIM_r$ of rank $r \geq 2$ was very recently determined \cite{CFGrowth}, where it was discovered that the exponential growth rate of $\FIM_r$ is an algebraic number $\alpha_r$. However, it would appear that free inverse monoids are better behaved with respect to growth than $\FrS$, as the monogenic free inverse monoid has cubic growth.

\section{The integral (co)homology of $\FrS$}\label{Sec:homology-of-FRS}

\noindent In this section we compute all integral (co)homology groups of the free regular $\star$-monoid $\FrS$ of any finite rank $r \geq 1$. To do this, we will construct the first terms of Kobayashi's resolution $(P_\ast, \partial_\ast)$, described in \S\ref{Subsec:kobayashi}), associated to the rewriting system $\cR_1$, defined by \eqref{Eq:R1-definition}, which is complete and defines $\FIS$. We then apply $- \otimes_{\Z \FIS} \Z$ to $(P_\ast, \partial_\ast)$ and construct a collapsing scheme for the resulting chain complex which removes all cells in dimension $3$ and above; this will yield a full description of the integral (co)homology groups of $\FIS$.

\subsection{Low-dimensional chain maps} For ease of notation, we will consider $\FIS$ as generated (as a monoid) by $\{ a, A \}$, where $A$ denotes $a^\star$. Then, in Kobayashi's resolution associated to $\cR_1$, we have $V^{(1)} = \{ a, A \}$, and furthermore we have that $V^{(2)}$ consists of two types of $2$-tuples:
\[
X_i = (a, a^{i-1} A^i a^i), \quad \text{and} \quad X_i^\star = (A, A^{i-1} a^i A^i) \quad (i \geq 1).
\]
It is easy to compute the action of $\partial_2$, which we record in the following lemma:

\begin{lemma}\label{Lem:partial-2-equation}
For all $i \geq 1$, we have 
\[
\partial_2(X_i) =  (a) \circ \sum_{s=0}^{i-1} a^s A^i a^i  + (A) \circ \sum_{s=0}^{i-1} A^s a^i.
\]
\end{lemma}
\begin{proof}
Here we use \eqref{Eq:partial_n-definition}, which yields that 
\begin{align*}
\partial_2(X_i) &= (a) \circ a^{i-1}A^i a^i - i_1 \left( \partial_1 (a) \circ a^{i-1} A^i a^i  \right) \\
&= (a) \circ a^{i-1}A^i a^i - i_1 \left( a^i - a^{i-1} A^i a^i \right)  \\
&= (a) \circ a^{i-1}A^i a^i - i_1(a^i)  + \left( (a) \circ \sum_{s=0}^{i-2} a^s A^i a^i \right) + \left( (A) \circ \sum_{s=0}^{i-1} A^s a^i \right) + i_1(a^i) \\
&= (a) \circ \sum_{s=0}^{i-1} a^s A^i a^i  + (A) \circ \sum_{s=0}^{i-1} A^s a^i,
\end{align*}
which is what was to be shown. 
\end{proof}
Of course, since $X_i$ is obtained from $X_i^\star$ by swapping all occurrences of $a$ for $A$, and vice versa, the analogous expression holds true for $\partial_2(X_i^\star)$. We now describe the action of $\partial_3$. It is easy to see, by considering all possible overlaps between rules in $\cR_1$, that any element of $V^{(3)}$ starting with $a$ is necessarily of one of the following three forms:
\begin{align}
&(X_i, A^i) \quad \text{where } 1 \leq i, \tag{I} \\
&(X_i, A^j a^j) \quad \text{where } 1 \leq j < i, \tag{II} \\
&(X_i, a^j A^k a^k) \quad \text{where } 1 \leq i \text{ and } 1 \leq j < k \leq i + j. \tag{III}
\end{align}
Notice that we have simplified the notation somewhat by writing e.g.\ $(X_i, A^i)$ rather than $(a, a^{i-1} A^i a^i, A^i)$. Analogously, we can obtain all elements of $V^{(3)}$ starting with $A$ by swapping all occurrences of $a$ for $A$ in the above sequences, and vice versa; this results in three other types of sequences, which we will refer to these as elements of type I$^\star$, II$^\star$, and III$^\star$, respectively. Clearly, by symmetry it is sufficient to compute the action of $\partial_3$ on the sequences starting with $a$. We do this now. 

\begin{lemma}\label{Lem:partial3-type1-tuples}
For all tuples $(X_i, A^i)$ of type I, we have 
\[
\partial_3(X_i, A^i) = X_i \circ A^i - X_i^\star.
\]
\end{lemma}
\begin{proof}
We have, by definition of $\partial_3$ as in \eqref{Eq:partial_n-definition}, that 
\begin{align*}
\partial_3(X_i, A^i) &= X_i \circ A^i - i_2 \partial_2 ( X_i \circ A^i) \\
&= X_i \circ A^i - i_2 \left( (a) \circ \sum_{s=0}^{i-1} \overline{a^s A^i a^i A^i} \right)  - i_2 \left( (A) \circ \sum_{s=0}^{i-1} \overline{A^s a^i A^i} \right) 
\end{align*}
using Lemma~\ref{Lem:partial-2-equation} and $\Z$-linearity of $i_2$. Now $\overline{a^s A^i a^i A^i} = a^s A^i$ is irreducible for all $s \geq 0$, and hence the first $i_2$-term vanishes identically. The term $A \cdot A^s a^i A^i$ is irreducible except when $s = i-1$, so the second term reduces to computing $i_2( (A)\circ A^{i-1} a^i A^i) = X_i^\star$. This yields the lemma. 
\end{proof}

\begin{lemma}\label{Lem:partial3-type2-tuples}
For all tuples $(X_i, A^j a^j)$ of type II, we have 
\[
\partial_3(X_i, A^j a^j) = X_i \circ (A^j a^j - 1).
\]
\end{lemma}
\begin{proof}
By definition of $\partial_3$ as in \eqref{Eq:partial_n-definition}, we have
\begin{align*}
\partial_3(X_i, A^j a^j) &= X_i \circ A^j a^j - i_2 \partial_2(X_i \circ A^j a^j) \\
&= X_i \circ A^j a^j - i_2 \left( (a) \circ \sum_{s=0}^{i-1} \overline{a^s A^i a^i A^j a^j} \right)  - i_2 \left( (A) \circ \sum_{s=0}^{i-1} \overline{A^s a^i A^j a^j} \right),
\end{align*}
where we have used Lemma~\ref{Lem:partial-2-equation}. Since $j < i$, we have that $\overline{a^s A^i a^i A^j a^j} = a^s A^i a^i$ and $\overline{A^s a^i A^j a^j} = A^s a^i$. Furthermore, for $0 \leq s \leq i-1$, these words remain irreducible when left multiplied by $a$ (resp.\ $A$), except the first when $s=i-1$. Thus, all terms in the argument of $i_2$ vanish except for $i_2((a) \circ a^{i-1} A^i a^i) = X_i$. This yields the lemma. 
\end{proof}

\begin{lemma}\label{Lem:partial3-type3-tuples}
For all tuples $(X_i, a^j A^k a^k)$ of type III, we have 
\[
\partial_3(X_i, a^j A^k a^k) = X_i \circ (a^j A^k a^k - a^j).
\]
\end{lemma}
\begin{proof}
As in the previous lemmas, we have 
\begin{align*}
\partial_3(X_i, a^j A^k a^k)&= X_i \circ a^j A^k a^k - i_2 \partial_2 ( X_i \circ a^j A^k a^k)
\end{align*}
and by Lemma~\ref{Lem:partial-2-equation} we have 
\[
i_2 \partial_2 ( X_i \circ a^j A^k a^k) = i_2 \left( (a) \circ \sum_{s=0}^{i-1} \overline{a^s A^i a^i a^j A^k a^k} \right)  - i_2 \left( (A) \circ \sum_{s=0}^{i-1} \overline{A^s a^i a^j A^k a^k} \right) 
\]
Since $i+j \geq k$, we have that $\overline{a^s A^i a^i a^j A^k a^k} = a^s A^i a^{i+j}$ and $\overline{A^s a^i a^j A^k a^k} = A^s a^{i+j}$. As in the previous two lemmas, irreducibility now implies that all $i_2$-terms vanish except one: the first, when $s=i-1$. Thus, we have
\[
i_2\partial_2(X_i \circ a^j A^k a^k) = i_2\left( (a) \circ a^{i-1} A^i a^{i+j} \right) = X_i \circ a^j,
\]
where the last equality is a routine computation. This yields the lemma. 
\end{proof}

Apply $- \otimes_{\Z\FIS} \Z$ to the Kobayashi resolution $(P_\ast, \partial_\ast)$, and denote the resulting complex by $(\widetilde{P}_\ast, \widetilde{\partial}_\ast)$. We can now compute the second homology group of $\FIS$. 

\begin{proposition}\label{Prop:H2-computation}
The monogenic free regular $\star$-monoid $\FIS$ has $H_2(\FIS, \Z) \cong \Z^\infty$. Consequently, $\FIS$ does not have the homological finiteness property $\FP_2$. 
\end{proposition}
\begin{proof}
By Lemma~\ref{Lem:partial-2-equation}, we have that $\ker \widetilde{\partial}_2$ is freely generated as a $\Z\FIS$-module by $\{ X_i - i X_1 \mid i \geq 1 \}$ and $\{ X_i - X_i^\star \mid i \geq 1 \}$, so $\ker \widetilde{\partial}_2$ is free abelian of countably infinite rank. By Lemmas~\ref{Lem:partial3-type1-tuples}, \ref{Lem:partial3-type2-tuples} and \ref{Lem:partial3-type3-tuples}, we have that $\widetilde{\partial}_3$ maps all basis elements of $\widetilde{P}_3$ to $0$ except those corresponding to type I tuples, in which case $\widetilde{\partial_3}( (X_i,A^i) \otimes_{\Z \FIS} 1)$ is just $(X_i - X_i^\star) \otimes_{\Z \FIS} 1$. Thus $\im \widetilde{\partial}_3$ has a basis $\{ (X_i - X_i^\star) \otimes_{\Z \FIS} 1 \mid i \geq 1\}$, so $\Tor_{2}^{\Z \FIS}(\Z, \Z) \cong \Z^\infty$, with a basis given by $\{ (X_i \otimes_{\Z \FIS} 1) +  \im \widetilde{\partial}_3  \mid i \geq 1\}$. 
\end{proof}

\begin{remark}
The above computations show that if we instead consider a ``bounded'' version $\FIS(n)$ of $\FIS$ in which we for some $n \geq 1$ only include the relations $a^i (a^\star)^i a^i = a^i$ and $(a^\star)^i a^i (a^\star)^i = (a^\star)^i$ for $1 \leq i \leq n$, then we have $H_2(\FIS(n), \Z) \cong \Z^{n-1}$. 
\end{remark}

\subsection{Collapsing scheme and higher homology groups} For brevity of notation, throughout this section given any basis element $v \in V^{(n)}$, we will also write $v$ for $v \otimes_{\ZM} 1$, as context makes this unambiguous. The basis for $\widetilde{P}_n$ will thus also be denoted $V^{(n)}$, and its elements will be referred to as cells. In this section, we will construct a matching function $\theta \colon \widetilde{P}_\ast \to \widetilde{P}_{\ast + 1}$ such that all basis elements in dimension $3$ are either redundant or collapsible. This will immediately imply that $H_k\FIS, \Z) = 0$ for $k \geq 3$. For ease of notation, throughout we will often suppress parentheses, e.g.\ substituting $\theta(x,y)$ for $\theta((x,y))$. 

For cells of dimension $n=1$, we set $\theta(a) = \theta(A) = 0$, and hence both such cells will be critical. For $n=2$, we set $\theta(X_i) = 0$ for all $i \geq 1$, and $\theta(X_i^\star) = (X_i^\star, a^i)$. Note that $\theta(X_i^\star)$ is thus a cell of type I$^\star$, and all such cells lie in the image of $\theta$. We now define $\theta$ when $n=3$. First, for cells of type I$^\star$, since $(X_i^\star, a^i) \in \im \theta$ we must have $(X_i^\star, a^i)$ be collapsible for all $i \geq 1$, so set $\theta(X_i^\star, a^i) = 0$. For the cells of type I, we set $\theta(X_i, A^i) = -(X_i, A^i, a^i A^i)$, the latter clearly being an element of $V^{(4)}$ since $(A^i, a^i A^i) \in \cE$. We thus need to have $(X_i, A^i, a^i A^i)$ to be collapsible, i.e.\ $\theta(X_i, A^i, a^iA^i) = 0$ and furthermore the coefficient of $(X_i, A^i)$ in $(\widetilde{\partial}_4 \theta)(X_i, A^i, a^i A^i)$ has to be $-1$. We now prove this is the case.

\begin{lemma}\label{Lem:theta-on-typeI-is-good}
For all $i \geq 1$, we have 
\[
\partial_4(X_i, A^i, a^i A^i) = (X_i, A^i) \circ a^i A^i + (X_i^\star, a^i) \circ A^i,
\]
and consequently $\widetilde{\partial}_4(X_i, A^i, a^i A^i) = \underbrace{(X_i, A^i)}_{\textnormal{redundant}} + \underbrace{(X_i^\star, a^i)}_{\textnormal{collapsible}}$. In particular, $(X_i, A^i)$ appears with the coefficient $-1$ in $(\widetilde{\partial}_4 \theta)(X_i, A^i)$.
\end{lemma}
\begin{proof}
By definition of $\partial_4$ and using the computation of $\partial_3$ on tuples of type I in Lemma~\ref{Lem:partial3-type1-tuples}, we have 
\begin{align*}
\partial_4(X_i, A^i, a^i A^i) &= (X_i, A^i) \circ a^i A^i - i_3 \partial_3\left( (X_i, A^i) \circ a^i A^i \right) \\
&= (X_i, A^i) \circ a^i A^i - i_3 \left( (X_i \circ A^i - X_i^\star) \circ a^i A^i \right)  \\
&= (X_i, A^i) \circ a^i A^i - i_3 \left( X_i \circ A^i \right) + i_3 \left( X_i^\star \circ a^i A^i \right),
\end{align*}
where we have used the fact that $\overline{A^i a^i A^i} = A^i$. We shall now see that there is plenty of cancellation in the terms of the argument of the $i_3$; indeed, we have 
\begin{align*}
i_3(X_i^\star \circ a^i A^i) &= (X_i^\star, a^i) \circ A^i + i_3 \left( i_2 \partial_2 \left( X_i^\star \circ a^i \right) \circ A^i \right)  \\
&= (X_i^\star, a^i) \circ A^i + i_3 \left( i_2 \left( \sum_{s=0}^{i-1} (A) \circ A^s a^i + \sum_{s=0}^{i-1} (a) \circ a^s A^i a^i \right) \circ A^i \right)  \\
&= (X_i^\star, a^i) \circ A^i + i_3 \left( i_2( (a) \circ a^{i-1}A^i a^i) \circ A^i\right) \\
&= (X_i^\star, a^i) \circ A^i + i_3(X_i \circ A^i),
\end{align*}
where the second equality follows from (the $\star$-dual of) Lemma~\ref{Lem:partial-2-equation} and rewriting the coefficients of $(A)$ and $(a)$; and the third equality follows from noting that whenever all $0 \leq s \leq i-1$, the word $A \cdot A^s a^i$ is always irreducible, and $a \cdot a^s A^i a^i$ is irreducible if and only if $s=i-1$. Thus we have cancellation of the $i_3(X_i \circ A^i)$-terms in the above expression for $\partial_4(X_i, A^i, a^i A^i)$, and we are left with the desired expression. 
\end{proof}

Thus we have extended our collapsing scheme in a consistent manner to the cells of type I and I$^\star$ in dimension $3$. All remaining cells in dimension $3$ will be chosen to be reducible. We proceed with the cells of type II and II$^\star$. For a given cell $(X_i, A^j a^j)$ of type II resp.\ a cell $(X_i^\star, a^j A^j)$ of type II$^\star$, where $1 \leq j < i$, we will set 
\begin{align*}
\theta(X_i, A^j a^j) &= -(X_i, A^j a^j, A^j), \\
\theta(X_i^\star, a^j A^j) &= -(X_i^\star, a^j A^j, a^j).
\end{align*}
We prove this assignment is consistent with setting $(X_i, A^j a^j, A^j)$ resp.\ $(X_i^\star, a^j A^j, a^j)$ to be collapsible. 

\begin{lemma}\label{Lem:theta-on-typeII-is-good}
For all $1 \leq j < i$, we have 
\[
\partial_4(X_i, A^j a^j, A^j) = (X_i, A^j a^j) \circ A^j,
\]
and hence $(X_i, A^j a^j)$ appears with the coefficient $-1$ in $\widetilde{\partial}_4 \theta (X_i, A^j a^j) = -(X_i, A^j a^j)$. 
\end{lemma}
\begin{proof}
By definition of $\partial_4$, it suffices to prove that $i_3 \partial_3 \left( (X_i, A^j a^j) \circ A^j \right) = 0$. But this is immediate by Lemma~\ref{Lem:partial3-type2-tuples}, since 
\[
\partial_3 \left( (X_i, A^j a^j) \circ A^j \right) = X_i \circ (A^j a^j - 1) a^j = X_i \circ (\overline{A^j a^j A^j} - A^j) = 0,
\]
which completes the proof of the lemma.
\end{proof}
Analogously, we can prove $\partial_4(X^\star_i, a^j A^j, a^j) = (X_i^\star, a^j A^j) \circ a^j$, thus showing that our assignment of all type II and II$^\star$-cells as redundant, together with the collapsible cells in dimension $4$ corresponding to them, is consistent with $\theta$ being a collapsing scheme.

Thus, it only remains to assign the cells of type III and type III$^\star$. We do this in a manner entirely analogous to the type II and II$^\star$. Namely, for a given cell $(X_i, a^j A^k a^k)$ of type III resp.\ a cell $(X_i^\star, A^j a^k A^k)$ of type III$^\star$, respectively, where $1 \leq j < k \leq i+ j$, we set
\begin{align*}
\theta(X_i, a^j A^k a^k) &= -(X_i, a^j A^k a^k, A^k), \\
\theta(X_i^\star, A^j a^k A^k) &= -(X_i^\star, A^j a^k A^k,a^k).
\end{align*}
We now prove that this assignment is consistent with setting $(X_i, a^j A^k a^k, A^k)$ resp.\ $(X_i^\star, a^j A^j, a^j)$ to be collapsible.

\begin{lemma}\label{Lem:theta-on-typeIII-is-good}
For all $1 \leq j < k + i+j$, we have 
\[
\partial_4(X_i, a^j A^k a^k, A^k) = (X_i, a^j A^k a^k) \circ A^k,
\]
and hence $(X_i, a^j A^k a^k)$ has coefficient $-1$ in $\widetilde{\partial}_4 \theta (X_i,a^j A^k a^k) = - (X_i, a^j A^k a^k)$. 
\end{lemma}
\begin{proof}
The proof is analogous to that of Lemma~\ref{Lem:theta-on-typeII-is-good}; indeed, it suffices to prove that $i_3 \partial_3( (X_i, a^j A^k a^k) \circ A^k) = 0$. But this follows directly from Lemma~\ref{Lem:partial3-type3-tuples}, since 
\[
\partial_3( (X_i, a^j A^k a^k) \circ A^k) = X_i \circ (a^j A^k a^k - a^j) A^k = X_i (\overline{a^j A^k a^k A^k} - a^j A^k) = 0,
\]
where we use the fact that $\overline{A^k a^k A^k} = A^k$. 
\end{proof}

Analogously, one can prove that $\partial_4(X_i^\star, A^j a^k A^k, a^k) = (X_i^\star, A^j a^k A^k) \circ a^k$, which together with Lemma~\ref{Lem:theta-on-typeIII-is-good} shows that our assignment of all type III and type III$^\star$-cells as redundant, together with the choice of their images under $\theta$ to be collapsible, is consistent with the definition of $\theta$ being a collapsing scheme. 

Thus, we have assigned all cells in dimension $3$ to be either redundant or collapsible; each collapsible cell in dimension $3$ is matched with a redundant cell in dimension $2$, and each redundant cell in dimension $3$ is matched with a collapsible cell in dimension $4$. For all remaining unassigned cells $v \in V^{(n)}$ in dimensions $n \geq 4$, we define $\theta(v)= 0$ and make all such cells critical. This completes the definition of $\theta$, and along the way we have proved:

\begin{proposition}\label{Prop:theta-is-collapsing-scheme}
The map $\theta$ is a collapsing scheme on $(\widetilde{P}_\ast, \widetilde{\partial}_\ast)$, and the Morse complex $(\widetilde{P}^\theta_\ast, \widetilde{\partial}^\theta_\ast)$ has no cells in dimension $3$.  
\end{proposition}

Thus by Proposition~\ref{Prop:Forman'sTheorem} we have that $H_\ast(\widetilde{P}^\theta_\ast) \cong H_\ast(\widetilde{P}_\ast) = H_\ast(\FIS, \Z)$, and since there are no critical cells in dimension $3$, there is simple chain homotopy between $(\widetilde{P}_\ast, \widetilde{\partial}_\ast)$ and truncating the complex in dimension $2$. Hence we have $H_k(\FIS, \Z) = 0$ for $k \geq 3$. We have thus completely determined the integral homology of the \textit{monogenic} free regular $\star$-monoid. We now extend this to all finitely generated free regular $\star$-monoids. To do this, we need a simple lemma. 

\begin{lemma}\label{Lem:free-products-good}
Let $X = X_0 \sqcup X_1$. Then $\Z\FSX(X)$ is free as a right $\Z\FSX(X_i)$-module.
\end{lemma}
\begin{proof}
We give an explicit basis for $\Z\FSX(X)$ as a right $\Z\FSX(X_1)$-module, which is sufficient for establishing the lemma. Let $B$ consist of all those $\mathbf{R}^\star(X)$-irreducible words $b$ over $X \cup X^\star$ such that (1) $b$ ends in a letter from $X_2 \cup X_2^\star$; and (2) $bw$ is irreducible for all words $w$ over $X_1 \cup X_1^\star$. By Proposition~\ref{Prop:Polak-prop}, we can identify $\mathbf{R}^\star(X)$-irreducible words with the element they represent in $\FSX(X)$, so we can speak of $B$ as a set of elements of $\FSX(X)$. We claim that $B$ is a basis for $\Z\FSX(X)$ as a right $\Z\FSX(X_1)$-module. 

First, we must show that $B$ is a generating set, for which it suffices to show that any element of $\FSX(X)$ can be written in the form $b w$ for some $b \in B$ and $w \in \FSX(X_1)$. Let $x \in \FSX(X)$ be any element. Since $\FSX(X)$ is the free product (in the category of regular $\star$-monoids) of $\FSX(X_1)$ and $\FSX(X_2)$, we can write $x = bw$ where $w \in \FSX(X_1)$ and $b$ are irreducible, and $b$ ends in a letter from $X_2 \cup X_2^\star$. Thus $b$ satisfies condition (1). If this $b$ does not satisfy condition (2), say $bw$ is reducible for some $w \in \FSX(X_1)$, then since all rules $(u \to v)$ of $\mathbf{R}^\star(X)$ end satisfy the condition that $u$ and $v$ end in the same letter, it follows that the unique irreducible descendant $y$ of $bw$ can be written of the form $b' w'$ for some $b'$ satisfying condition (1) and $w' \in \FSX(X_1)$. Furthermore, since $\mathbf{R}^\star(X)$ is length-reducing, and $b$ and $w$ are both irreducible, it follows that $|b'| < |b|$. If now $b'$ does not satisfy condition (2), then we can iterate the procedure, and by induction on the length of the word this eventually terminates in a word $b''$ satisfying condition (1) and (2). Thus we have proved that $B$ is a generating set. 

To prove uniqueness, note that any suffix of an element in $B$ is also an element of $B$. Hence, if $bw = b' w'$ for some $b, b' \in B$ and $w, w' \in \FSX(X_1)$, then by completeness of $\mathbf{R}^\star(X)$ either we have $bw$ graphically equal to $b'w'$, from which we immediately have $b = b'$ and $w = w'$, or else at least one of $bw$ and $b'w'$ is reducible; suppose, without loss of generality, that $bw$ is reducible. Then since $b$ and $w$ are irreducible, there must be some suffix $b''$ of $b$ and a prefix $w''$ of $w$ such that $b''w''$ is reducible; but $b'' \in B$, being a suffix of $b$, and clearly $w'' \in \FSX(X_1)$, contradicting property (2) of $b''$. Thus the decomposition is unique, and $B$ is a basis for $\Z\FSX(X)$ as a $\Z\FSX(X_1)$-module. 
\end{proof}

Using this, we can now obtain the main theorem of this section:

\begin{theorem}\label{Thm:main-homology-theorem}
Let $r \geq 1$, and let $\FrS$ be the free regular $\star$-monoid of rank $r$. Then:
\[
H_k(\FrS,\Z) \cong H^k(\FrS, \Z) \cong \begin{cases*}
\Z^r & if $k=1$, \\
\Z^\infty & if $k=2$, \\
0 & if $k \geq 3$.
\end{cases*}
\]
\end{theorem}
\begin{proof}
First, note that we have already proved the theorem for $r=1$ via Proposition~\ref{Prop:H2-computation} (which gives the second homology group) and Proposition~\ref{Prop:theta-is-collapsing-scheme} (which shows that the higher homology groups vanish). Also, for any monoid $M$, it is not difficult to show that $H_1(M, \Z) \cong H_1(G, \Z)$, where $G$ is the maximal group image of $G$. Since the maximal group image of $\FrS$ is clearly the free group on $r$ generators, we have $H_1(\FrS, \Z) \cong \Z^r$.

Thus, it remains to prove the theorem for the higher homology groups of $\FrS$ when $r>1$. Let $M_1, M_2$ be two copies of $\FIS$, let $M = M_1 \ast M_2 \cong \mathbf{F}_2^\star$, and let $(P^{(i)}_\ast, \partial^{(i)}_\ast)$ be a free $\ZM_i$-resolution of $\Z$ for $i=1,2$. By Lemma~\ref{Lem:free-products-good}, the functor $- \otimes_{\ZM_i} \ZM$ is exact, and hence $P^{(i)}_j \otimes_{\ZM_i} \ZM$ is a free $\ZM$-module for all $j \in \N$ and $i = 1, 2$. Note that there is a short exact sequence 
\[
0 \longrightarrow \ZM \xlongrightarrow{i} (\Z \otimes_{\ZM_1} \ZM) \oplus (\Z \otimes_{\Z M_2} \ZM) \xlongrightarrow{p} \Z \longrightarrow 0
\]
with $i$ given by the natural inclusion $m \mapsto (1 \otimes_{\ZM_1} m) \oplus (1 \otimes_{\ZM_2} m)$ and $p$ defined by $p(1 \otimes_{\ZM_1} m) = 1$ and $p(1 \otimes_{\ZM_2} m) = -1$. Thus, defining 
\[
Q_j =  (P^{(1)}_j \otimes_{\ZM_1} \ZM) \oplus (P^{(2)}_j \otimes_{\ZM_2} \ZM)
\]
for $j \geq 1$, and letting $\widetilde{\partial}_j = (\partial_j^{(1)} \otimes 1, \partial_j^{(2)} \otimes 1)$, a standard argument via the Nine Lemma (cf.\ e.g.\ \cite[Theorem~2.3]{Polyakova2007}) shows that there is an exact sequence
\[
\cdots \xlongrightarrow{\widetilde{\partial}_3} Q_2 \xlongrightarrow{\widetilde{\partial}_2} Q_1 \longrightarrow \ZM \longrightarrow \Z \longrightarrow 0 
\]
and since the $Q_i$ are all free (since $- \otimes_{\ZM_i} \ZM$ is exact), this is a free $\ZM$-resolution of $\Z$. Furthermore, it is clear from the description of $\widetilde{\partial}_{j}$ that $\Tor_{k}^{\ZM}(\Z, \Z) \cong \Tor_k^{\ZM_1}(\Z, Z) \oplus \Tor^{\ZM_2}_k(\Z, \Z)$ for $k \geq 2$. Thus, proceeding by induction on the number of free factors in $\FrS$, we find that $H_2(\FrS, \Z) \cong \Z^\infty$ and $H_k(\FrS, \Z) = 0$ for all $k \geq 3$. Finally, by the universal coefficient theorem it follows that $H_k(\FrS, \Z) \cong H^k(\FrS, \Z)$ for all $k \geq 1$, since all the homology groups of $\FrS$ are free. \end{proof}

Since $H_2(\FrS, \Z)$ is not finitely generated, we have the following consequence:

\begin{corollary}\label{Cor:Fp2}
No free regular $\star$-monoid is $\FP_2$. 
\end{corollary}

This corollary can also be obtained from Proposition~\ref{Prop:H2-computation} by noting that $\FrS$ retracts onto $\FIS$, and Pride \cite[Theorem~3]{Pride2006} has proved that the class of (left) $\FP_n$-monoids is closed under retractions for all $n \geq 1$. We remark that Gray \& Steinberg \cite{Gray2021} recently proved that no free inverse monoid of rank $r \geq 1$ is $\FP_2$, using rather different methods. Thus our Corollary~\ref{Cor:Fp2} is the regular $\star$-analogue of their result. Their result strengthened a classical result of Schein \cite{Schein1975} that free inverse monoids are not finitely presented as monoids. Jung-won Cho (personal communication) has shown that Schein's proof can be directly adapted to prove that $\FrS$ is not finitely presented; the above Corollary~\ref{Cor:Fp2} thus also implies this result (a proof via the rewriting system $\mathbf{R}^\star(X)$ can also be given). We remark finally that we do not know the integral homology groups of free inverse monoids, but that it would be interesting to know them. We suspect that they are all zero in dimension $2$ and higher.

Another consequence of Theorem~\ref{Thm:main-homology-theorem} is that if $A$ is any trivial $\Z\FrS$-module, then for all $k \geq 3$ we have
\[
H^k(\FrS, A) \cong \underbrace{\Hom(H_k(\FrS, \Z), A)}_{=0} \oplus \underbrace{\Ext^\Z_1(H_{k-1}(\FrS, \Z), A)}_{=0 \text{ since } \Ext^\Z_1(\Z, A) = 0} = 0
\]
by the universal coefficient theorem. Recall that the \textit{trivial cohomological dimension} $\operatorname{tcd}(M)$ of a monoid $M$ is the least integer $k$ such that $H^q(M,A) = 0$ for all $q > k$ and all trivial $\ZM$-modules $A$. Thus one corollary to Theorem~\ref{Thm:main-homology-theorem} is:

\begin{corollary}
The free regular $\star$-monoid $\FrS$ has trivial cohomological dimension $2$. 
\end{corollary}

The collapsing scheme constructed for the integral homology of $\FrS$ does not seem to extend to all other coefficients. It would therefore be interesting to know the cohomological dimension of $\FrS$, and whether or not it is also $2$. 

\section*{Acknowledgements}

I would like to thank Jung-won Cho, James East, Ganna Kudryavtseva, and Nik Ru\v{s}kuc for many helpful discussions about free regular $\star$-monoids and many adjacent topics.

\bibliography{freestar1-rewritten.bib}
\bibliographystyle{plain}
 
 \end{document}